\documentclass[hidelinks, 11pt]{article}

\usepackage{amsmath}					% AMS-wiskundige symbolen
\usepackage{amssymb} 					% invoegen van reele getallen etc
\usepackage{amsthm} 					% for \newtheorem commands
\usepackage[utf8]{inputenc} % use UTF-8, to make life easier
\usepackage{hyperref}
\usepackage{xcolor}
\usepackage{graphicx}
\usepackage{url}

\usepackage[T1]{fontenc}
\usepackage{libertine}

\def\titlerunning#1{\gdef\titrun{#1}}
\makeatletter
\def\author#1{\gdef\autrun{\def\and{\unskip, }#1}\gdef\@author{#1}}
\def\address#1{{\def\and{\\\hspace*{18pt}}\renewcommand{\thefootnote}{}%
		\footnote {#1}}%
	\markboth{\autrun}{\titrun}}
\makeatother
\def\email#1{\hspace*{4pt}{\em e-mail}: #1}
\def\MSC#1{{\renewcommand{\thefootnote}{}%
		\footnote{\emph{Mathematics Subject Classification (2010):} #1}}}
\def\keywords#1{\par\medskip
	\noindent\textbf{Keywords:} #1}

\newcommand{\verz}[1]{\left\{ #1 \right\}}

\newcommand{\Fq}{{\mathbb{F}_q}}
\newcommand{\F}{{\mathbb{F}}}

\newcommand{\UP}{\mathrm{UP}(q^2)}
\newcommand{\DUP}{\mathrm{DUP}(q^2)}
\newcommand{\PG}{\mathrm{PG}(2,q^2)}
\newcommand{\PGL}{\mathrm{PGL}}
\newcommand{\PGU}{\mathrm{PGU}}

\newcommand{\cD}{\mathcal{D}}

\newcommand{\cU}{\mathcal{U}}
\newcommand{\cP}{\mathcal{P}}
\newcommand{\cF}{\mathcal{F}}
\newcommand{\cL}{\mathcal{L}}
\newcommand{\cO}{\mathcal{O}}
\newcommand{\cG}{\mathcal{G}}

\newcommand{\cX}{\mathcal{X}}
\newcommand{\cZ}{\mathcal{Z}}
\theoremstyle{plain}
\newtheorem{theorem}{Theorem}[section]
\newtheorem{prop}[theorem]{Proposition}
\newtheorem{cor}[theorem]{Corollary}
\newtheorem{lemma}[theorem]{Lemma}

\theoremstyle{definition}
\newtheorem{defin}[theorem]{Definition}
\newtheorem{remark}[theorem]{Remark}

\newtheorem{question}[theorem]{Question}
\newtheorem{problem}{Open problem}

\begin{document}		
	
	{\titlerunning{}
	
	\title{Triangle-free induced subgraphs of the unitary polarity graph}
	
	\author{Sam Mattheus
		\and
		Francesco Pavese}
	
	\date{}
	
	\maketitle
	
	\address{S. Mattheus: Department of Mathematics, Vrije Universiteit Brussel, Pleinlaan 2, 1050 Brussel, Belgium; \email{sam.mattheus@vub.ac.be} 
		\and 
		F. Pavese: Dipartimento di Meccanica, Matematica e Management, Politecnico di Bari, Via Orabona 4, 70125 Bari, Italy; \email{francesco.pavese@poliba.it}}
	
	\MSC{Primary 05C50; Secondary 05B25 05C69 05C35}}

	{\renewcommand{\thefootnote}{}%
		\footnote{\textregistered \, 2018. Licensed under the Creative Commons CC-BY-NC-ND 4.0 \\ http://creativecommons.org/licenses/by-nc-nd/4.0/ }}
	
	\begin{abstract} 
		Let $\perp$ be a unitary polarity of a finite projective plane $\pi$ of order $q^2$. The unitary polarity graph is the graph with vertex set the points of $\pi$ where two vertices $x$ and $y$ are adjacent if $x \in y^\perp$. We show that a triangle-free induced subgraph of the unitary polarity graph of an arbitrary projective plane has at most $(q^4+q)/2$ vertices. When $\pi$ is the Desarguesian projective plane $\PG$ and $q$ is even, we show that the upper bound is asymptotically sharp, by providing an example on $q^4/2$ vertices. Finally, the case when $\pi$ is the Figueroa plane is discussed.
	\end{abstract}	
	
	\section{Introduction}
	
	Let $\pi$ be a finite projective plane of order $q$. A polarity $\perp$ of $\pi$ is an involutory bijective map sending points to lines and lines to points which preserves incidence. A point $x$ of $\pi$ is said to be absolute if $x \in x^\perp$. The \textbf{polarity graph} $\cG(\pi,\perp)$ is the graph with vertex set the points of $\pi$ where two vertices $x$ and $y$ are adjacent if $x \in y^\perp$. Remark that we could have defined the graph $\cG(\pi,\perp)$ equivalently with the lines of $\pi$ as vertices. This graph is not simple: every absolute point gives rise to a loop. A classical theorem by Baer \cite{B} states that every polarity has at least $q+1$ absolute points, which implies that there a polarity graph has at least $q+1$ loops. With a slight abuse of notation we will identify the vertices of the polarity graph with the points (or lines) of the plane. We will say for example that a point $x$ is adjacent to another point $y$. For all definitions and notions regarding projective planes and polarities not mentioned in Section \ref{section1}, we refer the reader to \cite{BE, JH, HP}.
	
	Polarity graphs and their properties have been the subject of study over the last few years. Questions regarding their independence number \cite{GN, MPS, MW}, chromatic number \cite{XTT} and other properties have been posed and (partially) answered. The motivation behind this line of research lies first of all in the fact that these graphs possess a lot of structure and interesting features. More importantly, polarity graphs are related to some classes of problems in extremal graph theory, among which Ramsey problems and Tur\'an-type problems. For example in the latter, F\"uredi \cite{F1, F2} has shown that the unique graph with the most edges among all graphs on $q^2+q+1$ vertices not containing $C_4$ as a subgraph is the polarity graph, where $\perp$ is an orthogonal polarity, i.e., a polarity with $q+1$ absolute points (which is the least possible as we already mentioned).
	
	Recently, Loucks and Timmons \cite{LT} have drawn attention to the following problem.
	
	\begin{question}
		What is the largest set of non-absolute vertices in $\cG(\pi,\perp)$ inducing a triangle-free subgraph?
	\end{question}
	
	Note that only non-absolute vertices are considered, since triangles in a polarity graph cannot contain absolute vertices, see also Section \ref{section1}.
	
	This problem first appeared in \cite{MW} in the context of extremal graph theory, where the authors considered the case when $\pi = \mathrm{PG}(2,q)$ and $\perp$ an orthogonal polarity. They used a construction due to Parsons \cite{P} to obtain an upper bound on the independence number of a 3-uniform hypergraph which first appeared in \cite{LV}. Parsons' construction on which they relied is exactly a triangle-free induced subgraph of $\cG(\pi,\perp)$. 
	
	Loucks and Timmons also mention that one of the motivations behind this question is from Tur\'an-type problems. In particular, we are interested in the maximum number of edges in an $n$-vertex graph without $C_3$ or $C_4$ as a subgraph. Indeed, it is natural to approach this problem by considering $C_4$-free graphs with many edges, and finding a $C_3$-free subgraph thereof.
	
	In this article, we investigate the case when $\perp$ is a unitary polarity. Then the order of the projective plane is necessarily a square, say $q^2$, there are $q^3+1$ absolute points and the set of absolute points forms a unital $\cU$. Note that there are unitals which do not arise from a unitary polarity, see \cite{BE} for further results on this topic. We denote by $\UP$ a unitary polarity graph for an arbitrary projective plane of order $q^2$. In the first part of the paper, by refining the techniques used in \cite{LT}, we obtain the following upper bound for a triangle-free induced subgraph of $\UP$. 
	
	\begin{theorem}\label{Th1}
	Let $S$ be a subset of non-absolute vertices of $\UP$ inducing a triangle-free subgraph, then 
	\[|S| \leq \frac{q^4+q}{2}.\]
	Moreover, if equality holds and $\ell$ is a line of $\pi$, then $|\ell \cap S| \in \{\frac{q^2-q}{2}, \frac{q^2+q}{2}\}$. 
	\end{theorem}
	
	In the second part of the paper we deal with the case when $\pi$ is the Desarguesian projective plane $\PG$. We will denote this graph by $\DUP$. When $q$ is even, we are able to show that the upper bound is asymptotically sharp.
	
	\begin{theorem}\label{Th2}
		For $q$ even, there exists a set of non-absolute vertices of $\DUP$ inducing a triangle-free subgraph of size $q^4/2$.
	\end{theorem}
	
	In the last part of the paper we consider the case when $\pi$ is the Figueroa plane $\cF$. The plane $\cF$ is obtained by the Desarguesian plane ${\rm PG}(2,q^3)$, by distorting certain lines. It is known that every polarity of $\cF$ induces a polarity of ${\rm PG}(2,q^3)$ \cite[Theorem 4.2]{Ham}. Vice versa, under certain assumptions, a polarity $\perp$ of the Desarguesian projective plane ${\rm PG}(2,q^3)$ gives rise to a polarity $\perp'$ of the Figueroa plane $\cF$. In this case, we show that a triangle-free induced subgraph of $\cG({\rm PG}(2,q^3), \perp)$ gives rise to a triangle-free induced subgraph of $\cG(\cF, \perp')$. This answers a question of Loucks and Timmons \cite[Question 1.4]{LT} in the case when $\pi$ is the Figueroa plane of order $q^6$.

	\section{Preliminaries about $\UP$}\label{section1}
	
	Before we can prove these results, we need some structural information about $\UP$, in particular about the neighbourhood structure, see \cite[Chapter XII]{HP}. If $x$ is a point of $\pi$, then $x^\perp$ denotes its {\em polar line}. Suppose first that $x$ is an absolute point, that is, its polar line contains $x$ itself. Therefore, $x$ is adjacent to $q^2$ vertices and has a loop. Let $y$ be a neighbour of $x$, then $x^\perp \cap y^\perp = \verz{x}$, which implies that $y$ has no neighbours in $N(x)$. This means that the subgraph induced by $x$ and its neighbours looks like a star. In particular, $x$ can never be contained in a triangle.
	
	On the other hand, if $x$ is a non-absolute point, then it is adjacent to $q^2+1$ other vertices. Among these there are $q+1$ absolute points, while the remaining $q^2-q$ are non-absolute points. Let $y$ be a non-absolute neighbour of $x$, then $x^\perp$ and $y^\perp$ intersect in a third point $z$. Hence, $x,y,z$ form a triangle in $\UP$. Moreover, $z$ is the unique common neighbour of $x$ and $y$. This implies that non-absolute neighbours of $x$ come in adjacent pairs, giving rise to $(q^2-q)/2$ triangles with common vertex $x$.
	
	A {\em self-polar triangle of $\pi$} (with respect to $\perp$) is a triangle each of whose vertices has the opposite side as polar line. From the discussion above it follows that triangles in $\UP$ are in one-to-one correspondence with self-polar triangles of $\pi$ and that there are exactly 
	$$	
	\frac{1}{3} (q^4-q^3+q^2) \frac{q^2-q}{2} = \frac{q^3(q^2-q+1)(q-1)}{6}
	$$
	of such triangles. Here and in the sequel we use the term {\em triangle} to refer to a triangle in $\UP$ or to a self-polar triangle of $\pi$.

	\section{The upper bound}
	
	In \cite{LT}, an upper bound for the number of vertices of $\UP$ inducing a triangle-free subgraph was proved. We refine their argument in order to obtain a better upper bound, see Theorem \ref{Th1}. To do so, we use techniques from spectral graph theory. 
	
	Given two subsets of vertices $S,T$ in a regular graph, let $e(S,T)$ denote the number of ordered pairs $(s,t)$, $s$ adjacent to $t$, where $s \in S, t \in T$. If $S = T$, then we simply write $e(S)$ instead of $e(S,T)$. Note that $e(S,T) = e(T,S)$. The following result, which first appeared in \cite[p17 Theorem 2.1.4]{Hthesis}, is known as the {\em expander mixing lemma} and it is a useful tool to estimate $e(S)$. Furthermore, it has found several applications in finite geometry over the last years \cite{BMS,LS,MW,WSV}.
	
	\begin{lemma}
		Let $G = (V,E)$ be a $d$-regular simple graph on $n$ vertices with eigenvalues $d = \lambda_1 \geq \lambda_2 \geq \dots \geq \lambda_n$. Let $\lambda = \max(|\lambda_2|,|\lambda_n|)$ be the second largest eigenvalue (in absolute value) and $S \subseteq V$, then the following inequality holds:
		\[\left\vert 2e(S)-\frac{d|S|^2}{n}\right\vert \leq \lambda|S|\left(1-\frac{|S|}{n}\right).\]
	\end{lemma}
	
	We will apply the expander mixing lemma to the graph $\Gamma$ obtained from $\UP$ by deleting its $q^3+1$ absolute points. Hence, $\Gamma$ is a simple $(q^2-q)$-regular graph on $q^4-q^3+q^2$ vertices. To apply the lemma, the second largest eigenvalue of $\Gamma$ is needed. This can be derived by using a technique due to Haemers, called {\em eigenvalue interlacing}. In what follows, we recall some definitions and results from \cite{H}.
	
	\begin{defin}
	Two sequences of real numbers $\lambda_1 \geq \dots \geq \lambda_n$ and $\mu_1 \geq \dots \geq \mu_m$ with $m < n$ \textbf{interlace} if
	\[\lambda_i \geq \mu_i \geq \lambda_{n-m+i} \text{ for } 1 \leq i \leq m.\]
	\end{defin}
		
	\begin{theorem} \label{interlacing1}
		Let $H$ be an induced subgraph of a graph $G$, then the eigenvalues of $H$ interlace those of $G$.
	\end{theorem}
	
	From \cite{GN}, the eigenvalues of $\UP$ are $q^2+1, q, -q$ with multiplicities $1, (q^4+2q^2-q)/2, (q^4+q)/2$,  respectively. Therefore, by Theorem \ref{interlacing1}, the eigenvalues of $\Gamma$, which are $q^2-q = \mu_1 \geq \dots \geq \mu_m$, with $m = q^4-q^3+q^2$, have to satisfy
	\[q = \lambda_2 \geq \mu_2 \geq \lambda_{q^3+3} = q \hspace{.2cm}\text{ and }\hspace{.2cm} -q =  \lambda_m \geq \mu_m \geq \lambda_n = -q.\]
	Either way, the second largest eigenvalue of $\Gamma$ (in absolute value) equals $q$.
	
	\begin{prop}\label{upper bound}
		Let $S$ be a subset of non-absolute vertices of $\UP$ inducing a triangle-free subgraph, then 
		\[|S| \leq \frac{q^4+q}{2}.\]
	\end{prop}
	\begin{proof}
		Applying the expander mixing lemma to $\Gamma$, we find
		\begin{align}\label{emlopnonabs}
		\left\vert 2e(S)-\frac{(q^2-q)|S|^2}{q^4-q^3+q^2}\right\vert \leq q|S|\left(1-\frac{|S|}{q^4-q^3+q^2}\right).
		\end{align}		
		Recall that every vertex $x$ is adjacent to $q^2-q$ vertices in $\Gamma$, which come in pairs to form $(q^2-q)/2$ triangles with common vertex $x$. Suppose that $x \in S$, then $x$ can be adjacent to at most one vertex of each triangle. This implies that $x$ has at most $(q^2-q)/2$ neighbours in $S$ and hence
		\[2e(S) = \sum_{x \in S}d_S(x) \leq \frac{q^2-q}{2}|S|,\]
		where $d_S(x)$ denotes the number of neighbours in $S$ of a vertex $x$. We can assume that $|S| \geq (q^4-q^3+q^2)/2$, otherwise the proposition is vacuously true. Then \eqref{emlopnonabs} becomes
		\begin{align*}
		\frac{(q^2-q)|S|^2}{q^4-q^3+q^2}-q|S|\left(1-\frac{|S|}{q^4-q^3+q^2}\right) \leq 2e(S) \leq \frac{q^2-q}{2}|S|.
		\end{align*}	
		Now solving the previous inequality for $|S|$ proves the result.
	\end{proof}
	
	The next step is to find out whether the upper bound can be attained. Assume that equality holds. From the proof of Proposition \ref{upper bound}, we have that each vertex $x \in S$ has degree $d_S(x) = (q^2-q)/2$. From a geometrical point of view, this means that if $x \in S$, then $|x^\perp \cap S| = (q^2-q)/2$, i.e., a certain number of lines intersect the set $S$ in a constant number of points. By again using eigenvalue interlacing, we will show that an even stronger property holds: for \textit{any} line $\ell$ we have \[|\ell \cap S| \in \verz{\frac{q^2-q}{2},\frac{q^2+q}{2}},\]
	i.e. $S$ is a {\em two-intersection set}. These point sets have been intensively studied in the literature, see for example \cite{CK, PR} and references therein. 
	
	\begin{defin}
		The interlacing of two sequences of real numbers $\lambda_1 \geq \dots \geq \lambda_n$ and $\mu_1 \geq \dots \geq \mu_m$, $m < n$, is \textbf{tight} if there exists $0 \leq k \leq m$ such that 
		\[\lambda_i = \mu_i \,\,\text{ for } 1 \leq i \leq k \text{ and } \lambda_{n-m+i} = \mu_i \,\,\text{ for } k+1 \leq i \leq m. \]
	\end{defin}
	
	\begin{theorem}\cite[p596 Corollary 2.3]{H}\label{tightinterlacing}	
		Let $A$ be a symmetric $n \times n$ matrix partitioned as
		\[A = \begin{pmatrix}
		A_{1,1} & \dots & A_{1,m} \\
		\vdots & & \vdots \\
		A_{m,1} & \dots & A_{m,m}
		\end{pmatrix},\]
		such that $A_{i,i}$ is a square matrix for all $1 \leq i \leq m$.
		The \textbf{quotient matrix} $B$ is the $m \times m$ matrix with entries the average row sums of the blocks of $A$. More precisely, 
		$$
		B = (b_{i,j}),  \; b_{i,j} = \frac{1}{n_i} {\bf 1}^t A_{i,j} {\bf 1} ,
		$$
		where ${\bf 1}$ denotes the all one column vector and $n_i$ is the number of rows of $A_{i,j}$. 
		Then the following holds
		\begin{enumerate}
			\item The eigenvalues of $B$ interlace those of $A$;
			\item if the interlacing is tight, then $A_{i,j}$ has constant row and column sums for $1 \leq i,j \leq m$.
		\end{enumerate}
	\end{theorem}

	\begin{lemma} 
		Let $S$ be a subset of non-absolute vertices of $\UP$ inducing a triangle-free subgraph, with $|S| = (q^4+q)/2$. Then, if $x$ is a point of $\pi$, we have that 
$$
|x^\perp \cap S| = 
\begin{cases}
\frac{q^2-q}{2} & \mbox{if} \; x \in S , \\ 
\frac{q^2+q}{2} & \mbox{if} \; x \notin S .
\end{cases} 
$$
	\end{lemma}
	\begin{proof}
	Let $A$ be the adjacency matrix of $\UP$. Recall that $\UP$ has loops, and the diagonal entries of $A$ are ones. We can partition the points of $\pi$ into three sets: the set of absolute points $\cU$, the set of interest $S$ and their complement $R$. Hence $|R| = q(q-1)(q^2-q+1)/2$. By considering the above partition, after reordering rows and columns, we get 
$$
A = \begin{pmatrix}
	A_{1,1} & A_{1,2} & A_{1,3} \\
	A_{2,1} & A_{2,2} & A_{2,3} \\
	A_{3,1} & A_{3,2} & A_{3,3} 
	\end{pmatrix}.
$$
Consider the quotient matrix 
$$
B = \begin{pmatrix}
	|\cU|^{-1} 2e(\cU) & |\cU|^{-1} e(\cU, S) & |\cU|^{-1} e(\cU, R) \\
	|S|^{-1} e(S, \cU) & |S|^{-1} 2 e(S) & |S|^{-1} e(S, R) \\
	|R|^{-1} e(R, \cU) & |R|^{-1} e(R, S) & |R|^{-1} 2 e(R)
	\end{pmatrix} .
$$ 
We already know a few entries. Indeed, by hypothesis, $|S|^{-1} 2 e(S) = (q^2-q)/2$. Moreover, every absolute point is adjacent with exactly one absolute point (itself), so $|\cU|^{-1} 2e(\cU) = 1$ and every non-absolute point is adjacent with $q+1$ absolute points, hence $|S|^{-1} e(S, \cU) = |R|^{-1} e(R, \cU) = q+1$. Analogously, since every point of $S$ is adjacent with exactly $(q^2-q)/2$ points of $R$ we get $|S|^{-1} e(S, R) = (q^2-q)/2$. As $e(\cU,S) = e(S,\cU)$, it follows that $|\cU|^{-1} e(\cU, S) = (q^2+q)/2$, and similarly $|\cU|^{-1} e(\cU, R) = (q^2-q)/2$ and $|R|^{-1} e(R, S) = (q^2+q)/2$. Lastly, since the sum of the elements of a row of $A$ equals $q^2+1$, we obtain $|R|^{-1} 2 e(R) = (q^2-3q)/2$. Collecting these values gives
	\[B = \begin{pmatrix}
	1 & \frac{q^2+q}{2} & \frac{q^2-q}{2} \\
	q+1 & \frac{q^2-q}{2} & \frac{q^2-q}{2} \\
	q+1 & \frac{q^2+q}{2} & \frac{q^2-3q}{2} 
	\end{pmatrix}.\]
	
	The eigenvalues of $B$ are $q^2+1,-q,-q$, which shows that the interlacing is tight. By Theorem \ref{tightinterlacing}, every block $A_{i,j}$ has constant row sum and constant column sum. This means that every vertex in $\cU$ or $R$ is adjacent to precisely $(q^2+q)/2$ vertices in $S$.
	\end{proof}

\begin{remark}
	Let $S$ be a subset of non-absolute vertices of $\UP$ inducing a triangle-free subgraph, with $|S| = (q^4+q)/2$. Then, in the language of \cite{DS}, we have that $S$ is an intriguing set of $\Gamma$, which could also be shown using Proposition 3.8 of that article.
\end{remark}

%	\begin{cor}\label{upperbound} 
%		Let $S$ be a subset of non-absolute vertices of $\UP$ inducing a triangle-free subgraph, then 
%		\[|S| < \frac{q^4+q}{2}.\]
%	\end{cor}
%	\begin{proof}
%	The set $S \cup \cU$ is a two-intersection set. Indeed, if $x \in S$, then $|x^\perp \cap S| = (q^2-q)/2$ and hence $|x^\perp \cap (S \cup \cU)| = (q^2+q)/2 +1$. On the other hand, if $x \notin S$, then $|x^\perp \cap S| = (q^2+q)/2$ and hence $|x^\perp \cap (S \cup \cU)|$ equals either $(q^2+q)/2 + 1$ or $(q^2+q)/2 + q + 1$, according as $x \in \cU$ or $x \notin \cU$. Then a double counting argument on (point,point)-line pairs $((P,Q), \ell)$, where $P,Q \in S \cup \cU$ and $\ell = PQ$ is a line of $\pi$ leads to the equation
%	\[k(k-1) = |S|\left(\frac{q^2+q+2}{2}\right)\left(\frac{q^2+q}2\right)+ (q^4+q^2+1-|S|)\left(\frac{q^2+3q+2}{2}\right)\left(\frac{q^2+3q}{2}\right),\] 
%	where $k = |S \cup \cU| = \frac{q^4+q}{2}+q^3+1$.\\
%	{\color{red} FINISH ARGUMENT, OR IS WRONG?}
	
%	\end{proof}	

\begin{remark}\label{strictinequalityforsmallq}
	Let $S$ be a subset of non-absolute vertices of $\UP$ inducing a triangle-free subgraph such that $|S| = (q^4+q)/2$. Then the set $S \cup \cU$ is a two-intersection set. Indeed, if $x \in S$, then $|x^\perp \cap S| = (q^2-q)/2$ and hence $|x^\perp \cap (S \cup \cU)| = (q^2+q)/2 +1$. On the other hand, if $x \notin S$, then $|x^\perp \cap S| = (q^2+q)/2$ and hence $|x^\perp \cap (S \, \cup \, \cU)|$ equals either $(q^2+q)/2 + 1$ or $(q^2+q)/2 + q + 1$, according as $x \in \cU$ or $x \notin \cU$. It follows that $S \cup \cU$ is a set of $(q+2)(q^3+1)/2$ points such that every line meets $S \cup \, \cU$ is either $(q^2+q+2)/2$ or $(q^2+3q+2)/2$ points. Since no such a set exists in ${\rm PG}(2,4)$ or in ${\rm PG}(2,9)$, see \cite{PR}, it follows that in these cases the upper bound of Proposition \ref{upper bound} cannot be attained. 
\end{remark}
	
	\section{The Desarguesian plane}
	
	Let $\pi$ be the Desarguesian projective plane $\PG$, with $q = p^h$, $p$ a prime, $h$ a positive integer. The set of absolute points of a unitary polarity of $\PG$ is called a {\em Hermitian curve}. In this case, if $q$ is even, by means of constructive arguments, we are able to show a lower bound close to the upper bound of Theorem \ref{Th1}. In particular we will prove the existence of a triangle-free subgraph of $\DUP$ having $q^4/2$ vertices, see Theorem \ref{Th2}. The strategy is the following: we will fix a unitary polarity $\perp$ and hence a Hermitian curve $\cU$; we will consider a set $\cP$ containing $q$ Hermitian curves such that $\cU$ belongs to $\cP$ and elements in $\cP$ pairwise intersect at a common point. Then we will select $q/2$ Hermitian curves in $\cP \setminus \{\cU\}$ and show that the set of points covered by these Hermitian curves distinct from their common point possesses the required properties.
	
	\subsection{A lower bound}  \label{ex}

	 The projective plane $\PG$ will be represented via homogeneous coordinates over the Galois field $\Fq$, i.e., represent the points of $\PG$ by $\langle (x, y, z)\rangle$, $x, y, z \in \Fq$, $(x, y, z) \ne (0, 0, 0)$, and similarly lines by $\langle [a, b, c] \rangle$, $a, b, c \in \Fq$, $[a, b, c] \ne [0, 0, 0]$. Incidence is given by $ax + by + cz = 0$. To avoid awkward notation the angle brackets will be dropped in what follows. The group consisting of all projectivities of $\PG$ is denoted by $\PGL(3,q^2)$. The point $U_i$ is the point with $1$ in the $i$-th position and $0$ elsewhere. As any two Hermitian curves are projectively equivalent \cite[Chapter 5]{JH}, we may assume that $\cU$ has equation
	$$
	X^qY+XY^q+Z^{q+1}=0.
	$$ 
	In other words, the matrix defining the polarity is the matrix
	\[P =\begin{pmatrix} 0 & 1 & 0 \\ 1 & 0 & 0 \\ 0 & 0 & 1 \end{pmatrix},\] two vertices $(x_1,y_1,z_1)$ and $(x_2,y_2,z_2)$ in $\DUP$ are adjacent if and only if
	\[\begin{pmatrix} x_1 & y_1 & z_1 \end{pmatrix}\begin{pmatrix} 0 & 1 & 0 \\ 1 & 0 & 0 \\ 0 & 0 & 1 \end{pmatrix}\begin{pmatrix} x_2^q \\ y_2^q \\ z_2^q \end{pmatrix} = 0 ,\] and $$(x, y, z)^\perp = [y^q, x^q, z^q] .$$ 
	Let $G \cong \PGU(3,q^2)$ be the subgroup of $\PGL(3,q^2)$ leaving $\cU$ invariant. We shall find it helpful to work with the elements of $\PGL(3,q^2)$ as matrices in ${\rm GL}(3,q^2)$ and the points of $\PG$ as column vectors, with matrices acting on the left.
	
	The point $U_2$ clearly belongs to $\cU$ and its polar line $U_2^\perp$ is the line $\ell : X = 0$. A pencil of Hermitian curves of $\PG$ generated by the Hermitian curves $\cal H, \cal H'$ with equations $F = 0$ and $F'= 0$, respectively, is the set of unitals defined by $\alpha F + \beta F'=0$, where $\alpha,\beta \in \Fq$, $(\alpha,\beta) \neq (0,0)$. We can consider the pencil $\cP$ generated by $\cU$ and the degenerate Hermitian curve $\ell$ defined by $X^{q+1}=0$. After normalization, this pencil contains the $q$ Hermitian curves defined by
	
	\[\cU_\lambda: \lambda X^{q+1} + X^qY+XY^q+Z^{q+1}=0,\]
	
	where $\lambda \in \Fq$. Clearly, when $\lambda =0$ we see that $\cU_0=\cU$, while for ``$\lambda = \infty$'' we retrieve the line $\ell$, which will be of lesser importance. Every point not on the line $\ell$ belongs to exactly one Hermitian curve of $\cP$, while the point $U_2$ is common to all Hermitian curves in $\cP$.

	\begin{lemma}\label{group}
	There exists a subgroup $K$ of $G$ of order $q^3$ acting regularly on points of \,$\cU_{\lambda} \setminus \{U_2\}, \lambda \in \Fq$. 
	\end{lemma}
	\begin{proof}
	Here, we shall consider the points of $\PG$ as column vectors, with matrices acting on the left. Let $K$ be the subgroup of $G$ whose elements are associated with the following matrices
	$$
	\begin{pmatrix}
	1 & 0 & 0 \\
	a & 1 & -b^q \\
	b & 0 & 1 
	\end{pmatrix}.
	$$
where $(1,a,b) \in \cU$. Then $K$ is a group of order $q^3$. Straightforward calculations show that if $P \in \cU_{\lambda}$ and $g \in K$, then $P^g \in \cU_{\lambda}$ and that the stabilizer in $K$ of a point $P \in \cU_{\lambda} \setminus \{U_2\}$ is trivial.     
	\end{proof}

	\begin{prop}\label{condition}
	Let $\lambda_1, \lambda_2, \lambda_3 \in \Fq \setminus \{0\}$, not necessarily distinct. If \, $\cU_{\lambda_1} \cup \, \cU_{\lambda_2} \cup \cU_{\lambda_3}$ contains a triangle, not containing the common absolute point $U_2$, then 
\begin{equation} \label{cond}
	\lambda_1\lambda_2 + \lambda_2\lambda_3 + \lambda_1\lambda_3 = 0 .
\end{equation} 
	\end{prop}
	\begin{proof}
	Suppose that we do have three points $P, Q, R$ forming a triangle and contained in $\cU_{\lambda_1} \cup \, \cU_{\lambda_2} \cup \, \cU_{\lambda_3}$. Taking into account Lemma \ref{group}, we may assume that $P$ is the point $(1,x_1,0) \in \cU_{\lambda_1}$, where	
	\begin{equation}\label{eq1} \lambda_1+x_1+x_1^q=0. \end{equation}
	The second point $Q = (1,x_2,y_2) \in \cU_{\lambda_2}$ has to be on the line $P^\perp :  x_1^q X + Y = 0$, which implies that $x_2 = -x_1^q$. Here we find by \eqref{eq1} that 	
	\begin{equation}\label{eq2} \lambda_2-x_1^q-x_1+y_2^{q+1}  = \lambda_1 + \lambda_2 +y_2^{q+1} = 0. \end{equation}
	For the third point $R$, we find in the same way that it should be of the form $(1,-x_1^q,y_3) \in \cU_{\lambda_3}$, where 
	\begin{equation}\label{eq3} \lambda_1 + \lambda_3 + y_3 ^{q+1} = 0.  \end{equation}
	Moreover, $R \in Q^\perp$, where $Q^\perp : -x_1 X + Y + y_2^q Z = 0$. This implies that 
	\begin{equation}\label{eq4} -x_1 - x_1^q+y_2^qy_3= \lambda_1+y_2^qy_3=0.  \end{equation}
	Remark that $y_2$ nor $y_3$ can equal zero, for otherwise one of $Q$ or $R$ would have to be $U_2$. Multiplying this last equation by $y_2y_3^q$ and using \eqref{eq2}, \eqref{eq3}, \eqref{eq4}, we obtain that
	$$
	\lambda_1y_2y_3^q + y_2^{q+1}y_3^{q+1} = \lambda_1(-\lambda_1^q)+(-\lambda_1-\lambda_2)(-\lambda_1-\lambda_3) =
	$$
	$$
	=-\lambda_1^2 + \lambda_1^2 + \lambda_1\lambda_2 + \lambda_2\lambda_3 + \lambda_1\lambda_3 = \lambda_1\lambda_2 + \lambda_2\lambda_3 + \lambda_1\lambda_3 = 0, 
	$$
	as $\lambda_1 \in \Fq$. 
	\end{proof}
	Therefore, if we can find a set $\Lambda \subseteq \Fq \setminus \{0\}$ such that for every $\lambda_1,\lambda_2,\lambda_3 \in \Lambda$ equation \eqref{cond} is never satisfied, then the set $\bigcup_{\lambda \in \Lambda} \cU_\lambda \setminus \{U_2\}$ induces a triangle-free subgraph on $|\Lambda| q^3$ vertices. We will call such a set $\Lambda$ a \textbf{good set}.

	\begin{remark}\label{oss}
	Note that, although Proposition \ref{condition} is useless in the case when $p = 3$, in all the other cases it provides the (weak) lower bound $q^3$ for the number of vertices of a triangle-free induced subgraph of $\DUP$. This can be achieved by considering the points of the unital $\cU_\lambda \setminus \verz{U_2}$, $\lambda \neq 0$.
	\end{remark}
	
	In order to find a good set $\Lambda$ of large size, we first restate the problem.
	
	\begin{lemma}\label{restatement}
		A good set $\Lambda$ is equivalent to a set $\overline{\Lambda} \subseteq \Fq \setminus \{0\}$ such that the sum of any three (not necessarily distinct) of its elements is never zero. In particular $|\Lambda| = |\overline{\Lambda}|$.
	\end{lemma}
	\begin{proof}
		Let $x, y, z \in \Lambda$. Then $xy + xz + yx \ne 0$ and dividing by $xyz$ we have that $x^{-1} + y^{-1} + z^{-1} \ne 0$. Let $\overline{\Lambda} := \{x^{-1} \; | \; x \in \Lambda\}$. Then the sum of any three (not necessarily distinct) elements of $\overline{\Lambda}$ is never zero. Vice versa, let $a, b, c \in \overline{\Lambda}$. Then $a + b + c \ne 0$ and dividing by $abc$ we have that $b^{-1} c^{-1} + a^{-1} c^{-1} + a^{-1} b^{-1} \ne 0$. If ${\Lambda} := \{x^{-1} \; | \; x \in \overline{\Lambda}\}$, then for any three (not necessarily distinct) elements of $\Lambda$, $\eqref{cond}$ is never satisfied.
	\end{proof}	
	
	We split into two cases, depending on the parity of $q$.
	
	\begin{lemma} \label{goodseteven}
	If $q$ is even, then there exists a good set $\Lambda \subseteq \Fq \setminus \{0\}$ of size $q/2$.	
	\end{lemma}
	\begin{proof}
	First we show the existence of an additive subgroup $H$ of size $q/2$ such that $1 \notin H$. Consider $\Fq$ as a vector space $V$ over $\mathbb{F}_2$, then the additive subgroups of $\Fq$ are in one-to-one correspondence with subspaces of this vector space $V$. In particular, a subgroup of size $q/2$ corresponds to a hyperplane of $V$. It is immediate that any of the $q/2$ hyperplanes not through the vector corresponding to $1$, gives rise to a subgroup $H$ satisfying all conditions. 
	
	Then we define the set $\overline{\Lambda}$ by
	\[\overline{\Lambda} = \{h+1 \,\, | \,\, h \in H\}.\]
	It is clear that sum of any three (not necessarily distinct) elements is never zero as
	\[(a+1)+(b+1)+(c+1)=0\]
	is equivalent with
	\[a+b+c=1,\]
	in contradiction with the choice of $H$. Therefore, we find by Lemma \ref{restatement} that $\Lambda = \{x^{-1} \; | \; x \in \overline{\Lambda}\}$ is a good set of size $q/2$.	
	\end{proof}
	
%	\begin{remark}
%	In the smallest case, taking into account Theorem \ref{Th1} and Remark \ref{strictinequalityforsmallq}, we may conclude that the largest triangle-free induced subgraph of ${\rm DUP}(4)$ has exactly $8$ vertices. 
%	\end{remark}
	This shows that for $q$ even, we have a triangle-free induced subgraph of $\DUP$ of size $q^4/2$, which proves Theorem \ref{Th2}.
	
	For $q$ odd, the situation is very different.

	\begin{remark}\label{remark}
	If $q$ is odd, let $P$ be a point of $\PG$ not in $\cU$ and let $T_P$ be the set of non-absolute points distinct from $P$ lying on the $q+1$ lines containing $P$ and tangent to $\cU$. Then $|T_P| = (q^2-1)(q+1)$. We claim that $T_P$ contains no triangle. Otherwise, if $Q_1, Q_2, Q_3$ were a triangle contained in $\cU$, we would obtain a configuration consisting of seven points: $P, Q_1, Q_2, Q_3, PQ_1 \cap \cU, PQ_2 \cap \cU, PQ_3 \cap \cU$ and seven lines $P^\perp, Q_1^\perp, Q_2^\perp, Q_3^\perp, PQ_1, PQ_2, PQ_3$ such that through each point there pass three lines and each line contains three points, i.e., a Fano plane ${\rm PG}(2,2)$. On the other hand, if $q$ is odd, ${\rm PG}(2,2)$ cannot be embedded in $\PG$. 
	Note that a larger set containing $T_P$ and not containing triangles can be obtained by adding to $T_P$ the $q^2-q$ non-absolute points on the line $P^\perp$. This gives the (weak) lower bound $q^3+2q^2-2q-1$ for the number of vertices of a triangle-free induced subgraph of $\DUP$, in the case when $q$ is odd.        
	\end{remark}

	A better lower bound than that of Remark \ref{remark} can be obtained in the case when $q$ is odd and $p \ne 3$. It relies on the following results, whose proofs were communicated to us by Bence Csajb\'ok \cite{BC} and by Anurag Bishnoi and Aditya Potukuchi \cite{BP} independently for the prime case.

	\begin{theorem}\label{goodsetodd}
	Let $q = p^h$ be odd, $3 \ne p = 3k \pm 1$. Then there exists a good set $\Lambda \subseteq \Fq \setminus \{0\}$ of size $kq/p$.
	\end{theorem} 
	\begin{proof}
	By Lemma \ref{restatement} it is sufficient to show the existence of a set $\overline{\Lambda} \subseteq \Fq \setminus \{0\}$ such that the sum of any three (not necessarily distinct) of its elements is never zero. Assume first that $h = 1$. If $p = 3k + 1$, then the sum of any three (not necessarily distinct) elements of $\overline{\Lambda} = \{1, 2, 3, \dots, k\}$ is non-zero. If $p = 3k - 1$, then the sum of any three (not necessarily distinct) elements of $\overline{\Lambda} = \{k, k + 1, \dots, 2k - 1\}$ is non-zero.

	Assume now that $h > 1$. Consider again $\Fq$ as a vector space $V$ over $\mathbb{F}_p$, then the additive subgroups of $\Fq$ are in one-to-one correspondence with subspaces of this vector space $V$. Let $A$ be any hyperplane of $V$ which does not contain $1$. Hence $A \cap \mathbb{F}_p = \{0\}$. Also, let $\bar{\Lambda}'$ be a subset of $\mathbb{F}_p$ of size $k$ such that the sum of any three (not necessarily distinct) of its elements is never zero. Put $\overline{\Lambda} = A + \bar{\Lambda}' := \{a + \lambda \; | \; a \in A, \lambda \in \bar{\Lambda}\}$. We claim that $\overline{\Lambda}$ has the required property. Indeed, if we had $(a_1 + \lambda_1) + (a_2 + \lambda_2) + (a_3 + \lambda_3) = 0$, for some $a_i \in A$ and $\lambda_i \in \bar{\Lambda}'$, $1 \le i \le 3$, then $-(a_1 + a_2 + a_3) = \lambda_1 + \lambda_2 + \lambda_3 \in \mathbb{F}_p$. However, $A \cap \mathbb{F}_p = \{0\}$, and thus $\lambda_1 + \lambda_2 + \lambda_3 = 0$, contradicting the choice of $\bar{\Lambda}'$.
	\end{proof}
	
	This good set for $q$ odd provides a triangle-free induced subgraph of size $q^4/3+o(q^4)$, which does not match the upper bound asymptotically.

	\subsection{Properties of the graph}
	
	Assume that $q$ is even and let $\Sigma$ be a triangle-free induced subgraph of $\DUP$ on $q^4/2$ vertices constructed in subsection \ref{ex}. Here, we investigate further properties of the graph $\Sigma$. To start off, we prove that $\Sigma$ is regular.
	
	\begin{prop}
		The graph $\Sigma$ is $q(q-1)/2$-regular.
	\end{prop}

	\begin{proof}
	Let  $P$ be a point of $\PG$, $P \notin \ell$, where $\ell$ is the tangent to $\cU$ at $U_2$. As we have partitioned all points of $\PG$ into the union of the $q$ sets $\cU_\lambda \setminus \{U_2\}$ and the line $\ell$, it is easy to see that the line $P^\perp$ contains a point of $\ell$, is secant to $q-1$ Hermitian curves of $\cP$ and is tangent to exactly one Hermitian curve of $\cP$. Consider a vertex $v \in \cU_\lambda \setminus \{U_2\}$, $\lambda \in \Lambda$ and let $\cU_{\mu}$ be the unique Hermitian curve of $\cP$ such that $|v^\perp \cap \cU_{\mu}| = 1$. Taking into account Lemma \ref{group}, we can assume that the point $v$ has coordinates $(1,x,0)$, where $x+x^q = \lambda$. Then $v^\perp$ has dual coordinates $[x^q,1,0]$ and we have to find $\mu$ such that $v^\perp$ is tangent to $\cU_\mu$. This means finding $\mu$ such that
		\begin{align*}
		0= \mu X^{q+1} + (x+x^q)X^{q+1}+Z^{q+1}= \mu X^{q+1} + \lambda X^{q+1}+Z^{q+1}
		\end{align*}
	has only one solution. Since $q$ is even, it follows immediately that $\lambda = \mu$. Hence, every vertex $v \in \cU_\lambda$ has exactly one neighbour in $\cU_\lambda$. Moreover, it has $q+1$ neighbours in $\Sigma$ on the $q/2-1$ other Hermitian curves $\cU_{\lambda}, \lambda \in \Lambda \setminus \{\mu\}$, which implies that the degree in $\Sigma$ of every vertex of $\Sigma$ is $q(q-1)/2$.	
	\end{proof}
	
	In fact, with a similar proof, one can show the exact intersection numbers for any line with $S$.
	
	\begin{cor}\label{intersection}
		Let $x$ be a point of $\PG$, then 
		$$
		|x^\perp \cap \Sigma| = 
		\begin{cases}
		\frac{q^2-q}{2} & \mbox{if} \; x \in \Sigma , \\ 
		\frac{q^2+q}{2} & \mbox{if} \; x \in \PG \setminus (S \cup \ell) , \\
		\frac{q^2}{2} & \mbox{if } \; x \in \ell \setminus \{U_2 \} , \\
		0 & \mbox{if } \; x = U_2 .
		\end{cases} 
		$$
	\end{cor}

	Remark that $q(q-1)/2$-regularity is the best we can achieve. Indeed, as we have already seen, every non-absolute point $v$ is adjacent to $q(q-1)$ other non-absolute points and these neighbours come in pairs to form $q(q-1)/2$ triangles with common vertex $v$, so $v$ can be adjacent to at most one vertex in each of these triangles. The fact that $\Sigma$ is $q(q-1)/2$-regular implies that $v$ is adjacent to \textit{exactly} one vertex in each of the triangles. In other words, if a triangle of $\DUP$ contains a vertex of $\Sigma$, it contains another vertex of $\Sigma$. Thus we have shown the following result.
	
	\begin{cor} 
		Every triangle of $\DUP$ has either $0$ or $2$ vertices in common with $\Sigma$.
	\end{cor}	
	
	This property allows us to show that the subgraph $\Sigma$ is maximal with the triangle-free property, i.e., we cannot add any vertex not in $\Sigma$ without creating a triangle.
	
	\begin{prop}
		The graph $\Sigma$ is maximal with respect to the triangle-free property.
	\end{prop}

	\begin{proof}
		Suppose we could add another vertex $v$. This vertex $v$ has at least one neighbour in $\Sigma$ as $v^\perp$ intersects any $\cU_{\lambda}$ in at least one point. Therefore, consider a triangle $T$ containing $v$ and a vertex of $v^\perp \cap \Sigma$. From the previous Corollary, we know that the triangle $T$ actually has its third vertex in $\Sigma$ and hence we cannot add $v$ to $\Sigma$ without creating a triangle.
	\end{proof}
	
The next result shows that $\Sigma$ can be chosen in such a way that it has girth $5$. Note that the following construction asymptotically matches the best known lower bound on the maximum number of edges in a $n$-vertex graph with girth at least five \cite{Bong}.
 
	\begin{prop}
		For $q$ an even prime power, there exists a $q(q-1)/2$-regular graph on $q^4/2$ vertices of girth $5$.
	\end{prop}
	
\begin{proof}
	We can show the result for $q \le 16$ using Magma \cite{magma}, so suppose $q \geq 32$ for the remainder of the proof. \\   
	Taking into account Lemma \ref{goodseteven}, let $\Lambda = \verz{1/(1+a) \; | \; a \in H},$ where $H$ is an additive subgroup of $\Fq$ of order $q/2$ not containing $1$. Let us consider a non-zero element $a \in H$ and let $\lambda_1 = 1/(1+a) \in \Lambda$. Let $b \in H$, with $a \ne b$ such that $b$ is not a solution of none of the following equations: 
	\begin{equation} \label{eq5}
		X^2+(a+1) X+a^3 = 0 , \;\; X^3+a X + a(a+1) = 0 , 
	\end{equation}
	\begin{equation} \label{eq6}
		X^2+(a+1) X+ a^2+a+1 = 0 . 
	\end{equation}
	Since the union of the solutions of the equations \eqref{eq5} and \eqref{eq6} consists of at most 7 distinct elements of $\Fq$, we can always find such an element $b$ if $q \ge 32$.		
	Let $P_1 := (1, x_1, 0) \in \cU_{\lambda_1}$ and $P_2 := (1, x_1^q, 0) = P_1^\perp \cap \cU_{\lambda_1}$ its unique neighbour on $\cU_{\lambda_1}$. Take another point $Q_1 := (1, x_1^q, z) \in P_1^\perp \cap \cU_{\lambda_2}$, with $\lambda_2 = 1/(1+b) \in \Lambda \setminus \{\lambda_1\}$. Hence $x_1 + x_1^q = \lambda_1$ and $z^{q+1} = \lambda_1 + \lambda_2$. Its unique neighbour on $\cU_{\lambda_2}$ is \hspace{.075em}$Q_2 := (1, x_1 + \lambda_1 + \lambda_2, z) = Q_1^\perp \cap \cU_{\lambda_2}$. Then $R := P_2^\perp \cap Q_2^\perp = (1, x_1, \lambda_2/z^q)$ belongs to $\cU_{\frac{\lambda_1^2 + \lambda_1 \lambda_2 + \lambda_2^2}{\lambda_1 + \lambda_2}}$. Note that, since $b$ is not a solution of \eqref{eq6}, we have that $\lambda_1^2 + \lambda_1 \lambda_2 + \lambda_2^2 \ne 0$. Hence $P_1 P_2 R Q_2 Q_1$ is a cycle of length $5$ in $\Sigma$ if and only if 
	\begin{equation} \label{eq7}
		\frac{\lambda_1^2 + \lambda_1 \lambda_2 + \lambda_2^2}{\lambda_1 + \lambda_2} \in \Lambda.
	\end{equation} 
	On the other hand, a straightforward calculation shows that \eqref{eq7} holds true if and only if 
		$$
		x:= \frac{a^2b + ab^2 + ab +1}{a^2 + b^2 + ab + a + b + 1} \in H .
		$$ 
	Note that $x \ne 1$, otherwise at least one among $a$ and $b$ should be $1$. Analogously to the proof of Lemma \ref{goodseteven}, view $\Fq$ as a vector space $V$ over $\mathbb{F}_2$. Since $b$ is not a solution of the equations \eqref{eq5}, we have that the vector subspace of $V$ generated by $a, b$ and $x$ does not contain $1$. Therefore, we can find a hyperplane $H$ such that $a,b,x \in H$ and $1 \notin H$. This hyperplane corresponds to an additive subgroup of size $q/2$, which concludes the proof. 
	\end{proof}

	\section{The Figueroa plane}

The finite {\em Figueroa planes} are non-Desarguesian projective planes of order $q^3$ for all prime powers $q > 2$. These planes were constructed algebraically in 1982 by Figueroa \cite{F}, and Hering and Schaeffer \cite{HS}, and synthetically in 1986 by Grundh\"ofer \cite{G}. All Figueroa planes of finite square order possess a unitary polarity and hence admit unitals \cite{DH}. It is known that every polarity of $\cF$ induces a polarity of ${\rm PG}(2,q^3)$ \cite[Theorem 4.2]{Ham}. Vice versa, it can be seen that under certain assumptions, a polarity $\rho$ of the Desarguesian projective plane ${\rm PG}(2,q^3)$  is ``inherited'' and gives rise to a polarity $\rho'$ of the Figueroa plane $\cF$. In this section we show that, in the case of ``inherited'' polarities, a triangle-free induced subgraph of $\cG({\rm PG}(2,q^3), \rho)$ gives rise to a triangle-free induced subgraph of $\cG(\cF, \rho')$.   

	\subsection{Construction of Figueroa planes}

Let $\alpha$ be an order $3$ collineation of the classical projective plane ${\rm PG}(2,q^3)$ of order $q^3$ over the finite field $\F_{q^3}$, where the fixed points of $\alpha$ constitute a subplane isomorphic to ${\rm PG}(2,q)$. The points and lines of ${\rm PG}(2,q^3)$ are partitioned into distinct types, as follows. A point $x$ of ${\rm PG}(2,q^3)$ belongs to $\cO_1$ if $x^\alpha = x$, or to $\cO_2$ if $x, x^\alpha, x^{\alpha^2}$ are distinct and on a line, or to $\cO_3$ if $x, x^\alpha, x^{\alpha^2}$ are distinct and not on a line. Types of lines $\cL_1, \cL_2, \cL_3$ of ${\rm PG}(2,q^3)$ are defined dually. Points of $\cO_1$ and lines of $\cL_1$ thus constitute a subplane $\cD$ isomorphic to ${\rm PG}(2,q)$. If $x$ is a point of $\cO_2$, then it is on the unique line of $\cL_1$ containing $x, x^\alpha, x^{\alpha^2}$. Conversely, if a point is on a unique line of $\cL_1$, then it belongs to $\cO_2$ since a point of $\cO_1$ is on $q+1$ lines of $\cL_1$ and a point of $\cO_3$ is on no line of $\cL_1$. It follows that if (and only if) a point is on no line of $\cL_1$, then it belongs to $\cO_3$.

Let $\mu$ be an involutory bijection between the points of $\cO_3$ and the lines of $\cL_3$ given as follows: if $x \in \cO_3$ and $\ell \in \cL_3$, then $x^{\mu} = x^\alpha x^{\alpha^2}$, and $\ell^{\mu} = \ell^\alpha \cap \ell^{\alpha^2}$. The Figueroa plane $\cF$ is obtained by the introduction of a new incidence between the set of points $\cP$ and the set of lines $\cL$ of ${\rm PG}(2,q^3)$, so that (viewing a line as a point set) the $q^3-q^2-q-2$ points of $\cO_3$ on the line $\ell = x^\alpha x^{\alpha^2} \in \cL_3$ distinct from $x^\alpha$ and $x^{\alpha^2}$ are replaced by other points of $\cO_3$ to form a new line. More precisely, as a set of points, a line of ${\rm PG}(2,q^3)$ belonging to $\cL_1$ or to $\cL_2$ remains unchanged as a line in $\cF$. As for a line $\ell = x^\alpha x^{\alpha^2} \in \cL_3$ in ${\rm PG}(2,q^3)$, where $x \in \cO_3$, let $y_i$, $1 \le i \le q^3 - q^2 - q - 2$, be the remaining points of $\cO_3$ on $\ell$. Consider the pencil of lines of $\cL_3$ on the point $x \in \cO_3$. Other than $x x^{\alpha}$ and $x x^{\alpha^2}$, the remaining lines of $\cL_3$ in the pencil are given by $z_i z_i^{\alpha}$,$1 \le i \le q^3 - q^2 - q - 2$, where each $z_i$ is a point of $\cO_3$. Let $\ell_{\cF}$ be the set of points obtained from $\ell$ by replacing each $y_i$ with $z^{\alpha^2}$. Then, $\ell_{\cF}$ is the Figueroa line corresponding to the line $\ell$ of $\cL_3$, see also \cite{Br}. Note that 
\[P \in \ell_{\cF} \cap \cO_3 \mbox{ if and only if }\ell^\mu \in P^\mu.\]

We observe the following property.
\begin{lemma}\label{lemma1}
	Let $\ell_1, \ell_2 \in \cL_3$ such that $\ell_1 \cap \ell_2 \in \cO_2$. Then the line $\ell_1^\mu \ell_2^\mu$ belongs to $\cL_2$.
\end{lemma}
\begin{proof}
	Assume by contradiction that $\ell_1^\mu \ell_2^\mu \in \cL_3$. Then $(\ell_1^\mu \ell_2^\mu)^\mu \in \cO_3$. Since $\ell_1^\mu \in (\ell_1^\mu \ell_2^\mu)$ and $\ell_2^\mu \in (\ell_1^\mu \ell_2^\mu)$, we have that $(\ell_1^\mu \ell_2^\mu)^\mu \in (\ell_1)_{\cF}$ and $(\ell_1^\mu \ell_2^\mu)^\mu \in (\ell_2)_{\cF}$. On the other hand $\ell_1 \cap \ell_2 = (\ell_1)_{\cF} \cap (\ell_2)_{\cF}$, since $\ell_1 \cap \ell_2 \in \cO_2$. It follows that $(\ell_1^\mu \ell_2^\mu)^\mu = \ell_1 \cap \ell_2$, a contradiction.  
\end{proof}

	\subsection{Polarities of $\cF$}

Let $\rho$ be a polarity of ${\rm PG}(2,q^3)$ such that $\rho$ and $\alpha$ commute, i.e., $\rho \alpha = \alpha \rho$. Let $\cX$ denote the set of $\rho$-absolute points. 
\begin{lemma}\label{lemma2}
The following properties hold true:
\begin{itemize}
\item[1)] $\cX$ is preserved by $\alpha$,
\item[2)] the point $P \in \cO_i$ if and only if $P^\rho \in \cL_i$,
\item[3)] the line $\ell \in \cL_i$ if and only if $\ell^\rho \in \cO_i$,  
\item[4)] if $P \in \cO_3$, then $P^{\mu \rho} = P^{\rho \mu}$,
\item[5)] if $\ell \in \cL_3$, then $\ell^{\mu \rho} = \ell^{\rho \mu}$.
\end{itemize}
\end{lemma}
\begin{proof}
Properties $1)$, $2)$ and $3)$ follow directly from the fact that the collineation $\alpha$ and the polarity $\rho$ commute. To prove $4)$, let $P$ be a point of $\cO_3$, then
$$
P^{\mu \rho} = (P^{\alpha} P^{\alpha^2})^{\rho} = P^{\alpha \rho} \cap P^{\alpha^2 \rho} = P^{\rho \alpha} \cap P^{\rho \alpha^2} = P^{\rho \mu}.
$$
Property $5)$ follows similarly.
\end{proof}

Consider the following map $\rho_{\cF}$: for points and lines of $\cO_1$ or $\cO_2$, $\rho_{\cF} = {\rho}$. For a point $x \in \cO_3$, $x^{\rho_{\cF}} = (x^{\rho})_{\cF}$, where $(x^{\rho})_{\cF}$ is the line of $\cF$ corresponding to the line $x^{\rho} \in \cL_3$ as described in the previous subsection. For a line $\ell \in \cL_3$, let $(\ell_{\cF})^{\rho_{\cF}} = \ell^{\rho}$. Since $\rho$ commutes with $\mu$, $\rho_{\cF}$ is indeed a polarity of $\cF$. Furthermore, if $x$ is a point of $\cO_3$, then $x$ is $\rho_{\cF}$-absolute if and only if $x^{\mu \rho} \in \cX$. Hence, if we denote by $\cX_{\cF}$ the $\rho_{\cF}$-absolute points, we have that 
$$
\cX_{\cF} = (\cX \cap \cO_1) \cup (\cX \cap \cO_2) \cup \{x^{\mu \rho} \; | \; x \in \cX \cap \cO_3\} .
$$
Since $\mu \rho$ is a bijection, the number of points of $\cO_3$ which are $\rho_{\cF}$-absolute equals the number of points of $\cO_3$ which are $\rho$-absolute. Thus, the number of absolute points of $\rho_{\cF}$ is the same as that of $\rho$.

A stronger result than Lemma \ref{lemma2} has been proved by Hamilton in his Ph.D thesis.
\begin{theorem}\cite[Theorem 4.2]{Ham}
Every polarity of the Figueroa plane $\cF$ induces a polarity of its Desarguesian subplane $\cD$ and vice versa.
\end{theorem}
Note that, since a polarity of $\cD$ extends to a polarity of ${\rm PG}(2,q^3)$, it easily follows that a polarity of $\cF$ induces a polarity of ${\rm PG}(2,q^3)$.

We end this section by considering the self-polar triangles with respect to polarities of $\cF$.

\begin{lemma}\label{triangles}
Let $T$ be a triangle containing at least two points of $\cO_1 \cup \cO_2$. Then T is self polar with respect to $\rho$ if and only if $T$ is self-polar with respect to $\rho_\cF$.
\end{lemma}

\begin{proof}
Let $T = \{P_1, P_2, P_2\}$, where $P_1, P_2 \in T \cap (\cO_1 \cup \cO_2)$. Then $P_1^{\rho} = P_1^{\rho_{\cF}} = P_2 P_3$ and $P_2^\rho = P_2^{\rho_{\cF}} = P_1 P_3$. On the other hand, $P_3^{\rho} = P_1 P_2$ if and only if $P_3 = P_1^\rho \cap P_2^\rho = P_1^{\rho_{\cF}} \cap P_2^{\rho_{\cF}}$ if and only if $P_3^{\rho_{\cF}} = P_1 P_2$, as required.
\end{proof}
\begin{remark}\label{oss1}
Note that if a triangle $T$ has at least one of its points in $\cO_1$, then $T$ is contained in $\cO_1 \cup \cO_2$, while if two of its points are in $\cO_1$, then its third point will belong to $\cO_1$ as well.
\end{remark}

\begin{lemma}
There is a bijection between the self-polar triangles of ${\rm PG}(2,q^3)$ with respect to $\rho$ and the self-polar triangles of $\cF$ with respect to $\rho_{\cF}$. 
\end{lemma}
\begin{proof}
Taking into account Lemma \ref{triangles} and Remark \ref{oss1}, we can consider the self-polar triangles containing no point of $\cO_1$ and at most one point of $\cO_2$. Let $T = \{P_1, P_2, P_3\}$ be a self-polar triangle of ${\rm PG}(2,q^3)$ with respect to $\rho$. Let $P_i^\rho = \ell_i$, that is, $\ell_1 = P_2 P_3$, $\ell_2 = P_1 P_3$, $\ell_3 = P_1 P_2$. Assume first that $T \subseteq \cO_3$. We show that $T$ is a self-polar with respect to $\rho$ if and only if $T_{\cF} = \{\ell_1^\mu, \ell_2^\mu, \ell_3^\mu\}$ is a self polar triangle with respect to $\rho_{\cF}$. Indeed, 
\begin{align*}
\ell_1^\mu = \ell_2^{\mu \rho_{\cF}} \cap \ell_3^{\mu \rho_{\cF}} & \iff \ell_{2}^{\mu}, \ell_{3}^{\mu} \in \ell_1^{\mu \rho_{\cF}} = (\ell_1^{\mu \rho})_{\cF} \\ 
& \iff (\ell_1^{\mu \rho})^{\mu} \in (\ell_2^{\mu})^\mu, (\ell_3^{\mu})^\mu \\ 
& \iff \ell_1^\rho \in \ell_2,\ell_3 \\ 
& \iff P_1 \in P_2^\rho, P_3^\rho \\
& \iff P_1 = P_2^\rho \cap P_3^\rho, 
\end{align*}

\noindent and similarly for any permutation of the indices.

On the other hand, if $P_1 \in \cO_2$ and $P_2, P_3 \in \cO_3$, let $P = \ell_2^{\mu \rho} \cap \ell_3^{\mu \rho} = (\ell_2^\mu \ell_3^\mu)^\rho = P_2^\mu \cap P_3^\mu$. Then, taking into account Lemma \ref{lemma1} and Lemma \ref{lemma2}, we have that $P \in \cO_2$. Moreover, $T$ is a self-polar triangle with respect to $\rho$ if and only $T_{\cF} = \{\ell_2^\mu, \ell_3^\mu, P\}$ is a self polar triangle with respect to $\rho_{\cF}$. Indeed, a similar argument as used above gives $P_2 = P_1^\rho \cap P_3^\rho$ if and only if $\ell_2^\mu = \ell_3^{\mu \rho_{\cF}} \cap P^{\rho_{\cF}}$ and $P_3 = P_1^\rho \cap P_2^\rho$ if and only if $\ell_3^\mu = \ell_2^{\mu \rho_{\cF}} \cap P^{\rho_{\cF}}$. Moreover, since $P \in \ell_i^{\mu \rho} \cap \cO_2 \subset \ell_i^{\mu \rho} \cap (\ell_i^{\mu \rho})_{\cF}$, $i = 2,3$, it follows that $P \in (\ell_i^{\mu \rho})_{\cF} = \ell_i^{\mu \rho_{\cF}}$.  
\end{proof}

\begin{theorem}\label{Th3}
Let $\cZ$ be a triangle-free set consisting of non-absolute points with respect to $\rho$, then 
$$
\cZ_{\cF} = (\cZ \cap \cO_1) \cup \{x^{\mu \rho} \; | \; x \in \cZ \cap \cO_3\}
$$
is a triangle-free set consisting of non-absolute points with respect to $\rho_{\cF}$.
\end{theorem}
\begin{proof}
Assume by contradiction that there exists a self-polar triangle with respect to $\rho_{\cF}$, say $T_{\cF}$, contained in $\cZ_{\cF}$, then necessarily $T_{\cF}$ is contained in $\cO_3$. If $T_{\cF} = \{P_1^{\rho \mu}, P_2^{\rho \mu}, P_3^{\rho \mu}\}$, then it follows that $T = \{P_1, P_2, P_3\}$ is a self-polar triangle with respect to $\rho$ contained in $\cZ$, a contradiction.   
\end{proof}

Finally, taking into account Theorem \ref{Th2} and Theorem \ref{Th3}, we have the following.

\begin{cor}
	Let $\rho$ be a unitary polarity of ${\rm PG}(2,q^6)$, $q$ even, and let $\alpha$ be an order 3 collineation of ${\rm PG}(2,q^6)$ fixing a subplane $\cD \cong {\rm PG}(2,q^2)$ pointwise, such that $\rho$ and $\alpha$ commute. Then, there exists a set of non-absolute vertices of $\cG(\cF, \rho_{\cF})$ inducing a triangle-free subgraph of size $\frac{q^{12}-q^4(q^4-1)(q^2+1)}{2}$. 
\end{cor}
\begin{proof}
	From Theorem \ref{Th2}, there exists a set $\cZ$ of non-absolute vertices of $\cG({\rm PG}(2,q^6), \rho)$ inducing a triangle-free subgraph of size $q^{12}/2$. From Theorem \ref{Th3}, the set $\cZ_{\cF}$ is a triangle-free set consisting of non-absolute points with respect to $\rho_{\cF}$, where $|\cZ_{\cF}| = |\cZ \setminus \cO_2|$. We will count the number of points of $\cZ \cap \cO_2$, which we have to remove, by inspecting the $q^4+q^2+1$ lines of $\cL_1$. Recall that these lines only contain points of $\cO_1$ and $\cO_2$. As every point of $\cO_2$ lies on exactly one line of $\cL_1$, we will find every point of $\cZ \cap \cO_2$ once. 
	
	By construction, $\cZ$ is the union of $q^3/2$ Hermitian curves of ${\rm PG}(2,q^6)$ pairwise meeting in a point $U_2$ of $\cD$ and having the same tangent line $\ell$ at $U_2$, with their common point $U_2$ deleted. Among these Hermitian curves there are $q/2$ meeting $\cD$ in a Hermitian curve of ${\rm PG}(2,q^2)$. One can see this by using the vector space representation of $\mathbb{F}_{q^3}$ over $\mathbb{F}_2$: the $q^3/2$ Hermitian curves are parametrized by elements of $\mathbb{F}_{q^3}$, which form a hyperplane. As $\mathbb{F}_q$ is a subspace of $\mathbb{F}_{q^3}$ not properly contained in the hyperplane, as it does not contain the element $1 \in \mathbb{F}_q$, this  means that it intersects $\mathbb{F}_q$ in $q/2$ points, which parametrize the $q/2$ Hermitian curves in $\cD$. It follows that for a line $r \in \cL_1$, we can compute  $|r \cap (\cZ \cap \cO_2)| = |r \cap \cZ|-|r \cap \cZ \cap \cO_1|$ using the intersection properties as stated in Corollary \ref{intersection} in both $\mathrm{PG}(2,q^6)$ and $\cD = \PG$ respectively. Therefore, if $r$ is a line of $\cO_1$, then 
	$$
	|r \cap \cZ|-|r \cap \cZ \cap \cO_1| =
	\begin{cases}
	\frac{q^6-q^3}{2}-\frac{q^2-q}{2} & \mbox{if} \; U_2 \notin r \text{ and } r^\rho \in \cZ, \\ 
	\frac{q^6+q^3}{2}-\frac{q^2+q}{2} & \mbox{if} \; U_2 \notin r \text{ and } r^\rho \notin \cZ, \\
	\frac{q^6}{2}-\frac{q^2}{2} & \mbox{if } \; U_2 \in r \neq \ell, \\
	0 & \mbox{if } \; r = \ell.
	\end{cases} 
	$$
	
	The number of lines corresponding to each case is respectively $q^4/2$, $q^4/2$, $q^2$ and $1$. Summing up over all these lines, we obtain the number $|\cZ \cap \cO_2|$ which we had to subtract from $q^{12}/2$ to obtain the result.
\end{proof}

\section{Conclusion and open problems}

In \cite{LT} the following question was posed.
	
\begin{question}\label{question}
	Given a finite projective plane $\pi$ of order $q$ and a polarity $\perp$, is it possible to find a triangle-free subgraph of the polarity graph of size $\frac{1}{2}q^2+o(q^2)$? 
\end{question}

When $\pi = \mathrm{PG}(2,q)$, this question has been almost completely resolved. Depending on the parity of $q$ and the type of $\perp$, there are four possibilities, shown in the table below.
\vspace{.5cm}
\begin{center}
\def\arraystretch{1.5}
\setlength\tabcolsep{.5cm}	
\begin{tabular}{| c | c | c |}
	\hline
	prime power $q$ & type & answer \\ \hline 
	even & pseudo & yes \cite{MPS} \\ %\hline
	odd & orthogonal & yes \cite{P} \\ %\hline
	even square & unitary & yes \\ %\hline
	odd square & unitary & ? \\ \hline
\end{tabular}
\end{center}
\vspace{.5cm}
	Starting from Question \ref{question}, we can state three open problems, ranked in what we believe to be increasing difficulty.
	
\begin{problem}
	Show that there exists a triangle-free induced subgraph of $\DUP$, $q$ odd, of size $\frac{1}{2}q^4+o(q^4)$.
\end{problem}

In Section 4.1 we mention the existence of a triangle-free induced subgraph of $\DUP$, $q$ odd, of size $q^4/3+o(q^4)$. Other ideas will be needed to find larger triangle-free subgraphs.

\begin{problem}[Conjecture 1 in \cite{MW}]
	Prove or disprove that in case 2, i.e. $\pi = \mathrm{PG}(2,q)$ and $\perp$ is an orthogonal polarity, Parsons' examples are the largest. If true, is it possible to show that they are the unique triangle-free induced subgraphs of this size?
\end{problem}	
	
As shown by Loucks and Timmons, this can only be true when $q$ is large enough. Using a computer search, they found larger examples for $q = 5,7,9,13$.

\begin{problem}
	What if $\pi$ is not the Desarguesian projective plane $\mathrm{PG}(2,q)$? Can we still answer Question \ref{question} in the affirmative?
\end{problem}		

In the case when $\pi$ is the Figueroa plane and the unitary polarity is inherited, we showed that the answer is indeed yes. In fact, one can do this for any inherited polarity by Theorem \ref{Th3}, but for this article, we restrict ourselves to the unitary case.

Lastly, remark that there exist projective planes of order $q$ and polarities where the size of the set of absolute points does not belong to $\verz{q+1,q\sqrt{q}+1}$, see \cite[Chapter XII]{HP}, \cite{Piper}. \\

\textbf{Acknowledgment.} We thank the referees for their valuable comments which improved the quality of this paper. We also thank B. Csajb\'ok for the proofs of Lemma \ref{restatement} and Theorem \ref{goodsetodd} and A. Bishnoi and A. Potukuchi for the proof of the prime case independently.
	
\urlstyle{same}


\begin{thebibliography}{SK}
		
	\bibitem{B} R. Baer, Polarities in finite projective planes, {\em Bull. Amer. Math. Soc.}, 52:77--93, 1946.
		
	\bibitem{BE} S. Barwick, G. Ebert, {\em Unitals in projective planes}, Springer Monographs in Mathematics, Springer New York, 2008.
		
	\bibitem{BMS} A. Bishnoi, S. Mattheus, J. Schillewaert, Minimal multiple blocking sets, {\em \url{https://arxiv.org/abs/1703.07843}}, 2017.
	
	\bibitem{BP} A. Bishnoi, A. Potukuchi, {\em Personal communication}.
	
	\bibitem{Bong} N.H. Bong, Properties and structures in extremal graphs, {\em PhD thesis}, University of Newcastle, Australia, 2017.
		
	\bibitem{magma} W. Bosma, J. Cannon, C. Playoust, The Magma algebra system. I. The user language. Computational algebra and number theory, {\em J. Symbolic Comput.}, 24:235--265, 1997.			
		
	\bibitem{Br} J.M.N. Brown, Some partitions in Figueroa planes, {\em Note Math.} 29:33--44, 2009.

	\bibitem{CK} R. Calderbank, W.M. Kantor, The geometry of two-weight codes, {\em Bull. London Math. Soc.}, 18:97--122, 1986.
	
	\bibitem{BC} B. Csajb\'ok, {\em Personal communication}.
	
	\bibitem{DS} B. De Bruyn, H. Suzuki, Intriguing sets of vertices of regular graphs, {\em Graphs Combin.}, 26:629--646, 2010. 
		
	\bibitem{DH} M. De Resmini, N. Hamilton, Hyperovals and unitals in Figueroa planes, {\em Eur. J. Combin.}, 19:215--220, 1998.

	\bibitem{F} R. Figueroa, A family of not $(V, l)$-transitive projective planes of order $q^3$, $q \not\equiv 1 \pmod 3$ and $q > 2$, {\em Math. Z.}, 181:471--479, 1982.

	\bibitem{F1} Z. F\"uredi, Graphs without Quadrilaterals, {\em J. Combin. Theory Ser. B}, 34:187--190, 1983.

	\bibitem{F2} Z. F\"uredi, On the Number of Edges of Quadrilateral--Free Graphs, {\em J. Combin. Theory Ser. B}, 68:1--6, 1996.

	\bibitem{GN} C.D. Godsil, M.W. Newman, Eigenvalue bounds for independent sets, {\em J. Combin. Theory Ser. B}, 98:721--734, 2008.
		
	\bibitem{G} T. Grundh\"ofer, A synthetic construction of the Figueroa planes, {\em J. Geom.}, 26, 191--201, 1986.

	\bibitem{Hthesis} W.H. Haemers, Eigenvalue techniques in design and graph theory, {\em PhD thesis}, University of Eindhoven, 1979.

	\bibitem{H} W.H. Haemers, Interlacing Eigenvalues and Graphs, {\em Linear Algebra Appl.}, 226/228:593--616, 1995.
	
	\bibitem{Ham} N. Hamilton, Maximal Arcs in Finite Projective Planes and Associated Structures in Projective Spaces, {\em PhD thesis}, University of Western Australia, 1995.
		
	\bibitem{HS} Ch., Hering, H.J. Schaeffer, On the new projective planes of R. Figueroa, {\em In: Jungnickel, D. et al. (eds.) Combinatorial Theory. Proc. Schloss Rauischholzhausen}, 1982, pp. 187--190. Springer, Berlin, 1982.

	\bibitem{JH} J.W.P. Hirschfeld, {\em Projective Geometries over Finite Fields}, Oxford Mathematical Monographs, Oxford Science Publications, The Clarendon Press, Oxford University Press, New York, 1998.
	
	\bibitem{HP} D. Hughes, F. Piper, {\em Projective planes}, Graduate Texts in Mathematics, Springer-Verlag New York-Berlin, 1973.

	\bibitem{LV} F. Lazebnik, J. Verstraete, On hypergraphs of girth five, {\em Electron. J. Combin.}, 10:25, 2003.
		
	\bibitem{LT} J. Loucks, C. Timmons, Triangle--free induced subgraphs of polarity graphs, {\em \url{https://arxiv.org/abs/1703.06347}}, 2017.
		
	\bibitem{LS} B. Lund, S. Saraf. Incidence bounds for block designs, {\em SIAM J. Discrete Math.}, 30:1997--2010, 2016.
		
	\bibitem{MPS} S. Mattheus, F. Pavese, L. Storme, On the independence number of graphs related to a polarity, {\em \url{https://arxiv.org/pdf/1704.00487}}, 2017.
		
	\bibitem{MW} D. Mubayi, J. Williford, On the independence number of the Erd\H os-Rényi and projective norm graphs and a related hypergraph, {\em J. Graph Theory}, 56:113--127, 2007.
		
	\bibitem{P} T.D. Parsons, Graphs from projective planes, {\em Aequationes Math.}, 14:167--189, 1976.
		
	\bibitem{PR} T. Penttila, G. Royle, Sets of type $(m,n)$ in the Affine and Projective Planes of Order Nine, {\em Des. Codes Cryptogr.}, 6:229--245, 1995.

	\bibitem{Piper} F. Piper, Polarities in the Hughes plane, {\em Bull. London Math. Soc.} 2:209--213, 1970.

	\bibitem{WSV} S. D. Winter, J. Schillewaert, and J. Verstraete, Large incidence--free sets in geometries, {\em Electron. J. Combin}, 19:24, 2012. 
		
	\bibitem{XTT} P. Xing, M. Tait, C. Timmons, On the chromatic number of the Erd\H os-Rényi orthogonal polarity graph, {\em Electron. J. Combin.}, 22:2.21, 2015.
		
	\end{thebibliography}
\end{document}